\documentclass{amsart}
\usepackage{amssymb}
\usepackage{graphicx}

\theoremstyle{plain}
\newtheorem{theorem}{Theorem}[section]
\newtheorem{corollary}[theorem]{Corollary}
\newtheorem{definition}[theorem]{Definition}
\newtheorem{proposition}[theorem]{Proposition}

\numberwithin{equation}{section}

\begin{document}

\title[Colored base-3 partitions]{Colored base-3 partitions, sequences of polynomials, and perfect numbers}

\author{Karl Dilcher}
\address{Department of Mathematics and Statistics\\
         Dalhousie University\\
         Halifax, Nova Scotia, B3H 4R2, Canada}
\email{dilcher@mathstat.dal.ca}
\author{Larry Ericksen}
\address{1212 Forest Drive, Millville, NJ 08332-2512, USA}
\email{LE22@cornell.edu}
\thanks{Research supported in part by the Natural Sciences and Engineering
        Research Council of Canada, Grant \# 145628481}

\date{}

\setcounter{equation}{0}

\begin{abstract}
Motivated by the observation that the counting function of a certain base-3
colored partition contains the even perfect numbers as a subsequence, we
begin by defining a sequence of polynomials in four variables and discuss their
properties and combinatorial interpretations. We then concentrate on certain
subsequences that are related to the Chebyshev polynomials of both kinds.
Finally, we consider several sequences of single-variable polynomials that have
meaningful combinatorial interpretations as well as interesting zero 
distributions.
 
\end{abstract}

\maketitle

\section{Introduction}\label{sec:1}

There are countless different ways of writing a positive integer as a sum of
other positive integers according to a fixed set of rules. The most famous 
among these are the classical partitions; they go back to Euler and are still 
an active area of research. Variants include compositions, overpartitions,
colored partitions, partitions with restrictions of various kinds, and binary
or base-$b$ analogues of all of the above. 

Let us begin with the following base-3 partition rule and the corresponding
counting function. Given an integer $n\geq 1$, in how many ways can we write
$n$ as a sum of powers of 3 under the following conditions?
\begin{equation}\label{1.1}
\begin{cases}
&\hbox{At most one of each power may be overlined},\\
&\hbox{at most one of each power may be marked by a tilde, and}\\
&\hbox{at most two of each power may remain unmarked.}
\end{cases}
\end{equation}
We also assume that at most one of an overline or a tilde are allowed, and that
$3^0=1$ is also an allowable power of 3. 

Such partitions are usually referred
to as restricted colored (base-$b$) partitions, where in this case the three
``colors" are: (1) ``marked by a line", (2) ``marked by a tilde", and 
(3) ``unmarked". Given the condition \eqref{1.1}, we say that the restriction 
is of type $(1,1,2)$.
We denote the number of partitions as given in \eqref{1.1} by $S(n)$ and 
illustrate them with two examples.

\medskip
\noindent
{\bf Example~1.} Let $n=3$. Then the allowable partitions are
\[
3,\quad\overline{3},\quad\tilde{3},\quad 1+1+\overline{1},
\quad 1+1+\tilde{1},\quad 1+\overline{1}+\tilde{1},
\]
and we have $S(3)=6$.

\medskip
\noindent
{\bf Example~2.} Let $n=12$. To proceed systematically, we first consider the 
unmarked and unrestricted base-3 partitions, written more concisely as
\begin{equation}\label{1.2}
(93), (9111), (3333), (333111), (331\ldots 1), (31\ldots 1), (1\ldots 1).
\end{equation}
The last three of these partitions have too many 1s to satisfy the conditions
in \eqref{1.1}. However, the first four in \eqref{1.2} lead to the 
following restricted colored partitions:
\begin{align*}
&(93),(9\overline{3}),(9\tilde{3}),
(\overline{9}3),(\overline{9}\overline{3}),(\overline{9}\tilde{3}),
(\tilde{9}3),(\tilde{9}\overline{3}),(\tilde{9}\tilde{3});\\
&(911\overline{1}), (911\tilde{1}), (91\overline{1}\tilde{1}),\ldots,
(\tilde{9}1\overline{1}\tilde{1});\\
&(33\overline{3}\tilde{3});\\
&(33\overline{3}11\overline{1}), (33\overline{3}11\tilde{1}),\ldots,
(33\overline{3}1\overline{1}\tilde{1}),
\end{align*}
where in the second and fourth subsets we only showed a few representative 
ones. The numbers of partitions in these four blocks are 9,9,1 and 9, 
respectively, for a total of $S(12)=28$.

\medskip
We notice that the number of partitions, 6 and 28, in these two examples are
perfect numbers. Of course, this could be just a coincidence, but it turns out
that 496 and 8128 also occur in this way, for different $n$. One purpose of 
this paper is to
explain this observation. However, apart from this rather unexpected connection,
the even perfect numbers occur only as a subsequence of another well-known 
sequence, as we will see.

In the process of explaining this, we will study several general and more 
specific sequences of polynomials and their combinatorial interpretations.
In Section~\ref{sec:2} we introduce a sequence of polynomials in four variables,
which turns out to be a polynomial analogue of the sequence $S(n)$ defined by
\eqref{1.1}. We also derive several properties of these polynomials, 
including combinatorial interpretations in terms of restricted base-3 
partitions.

In Section~\ref{sec:3} we define two different but related subsequences of 
polynomials and derive several properties, including recurrences and 
connections with the Chebyshev polynomials of both kinds. We also explain the
connection with perfect numbers. In Sections~\ref{sec:4} and~\ref{sec:5} we
consider several sequences of single-variable polynomials that arise from 
specializing or equating certain of the four variables in the general case.
These polynomial sequences have interesting zero distributions and specific
combinatorial interpretations. We conclude this paper with some further 
remarks in Section~\ref{sec:6}.

\section{Some generalities}\label{sec:2}

The concept of colored (or multicolor) partitions is relatively recent; it was
introduced by Keith \cite{Ke} as a generalization of overpartitions. Colored 
$b$-ary (or base-$b$) partitions were more recently studied by Ulas and 
{\.Z}mija \cite{UZ}. In \cite{DE10} we developed a general theory of restricted
multicolor $b$-ary partitions, along with polynomial analogues and 
characterizations of individual partitions. A more complete list of references
can also be found in \cite{DE10}.

Here, we are going to consider the following special case. We begin with a 
definition.

\begin{definition}
{\rm Let $Z:=(w,x,y,z)$ be a quadruple of real or complex variables. We define 
the sequence $S(n;Z)$ of polynomials by the generating function}
\begin{equation}\label{2.1}
\sum_{n=0}^{\infty}S(n;Z)q^n = \prod_{j=0}^{\infty}\left(1+wq^{3^j}\right)
\left(1+xq^{3^j}\right)\left(1+yq^{3^j}+zq^{2\cdot 3^j}\right).
\end{equation}
\end{definition}

The first few of these polynomials are shown in Table~1.

\bigskip
\begin{center}
{\renewcommand{\arraystretch}{1.1}
\begin{tabular}{|r|l|}
\hline
$n$ & $S(n;Z)$ \\
\hline
0 & 1 \\
1 & $w+x+y$ \\
2 & $wx+wy+xy+z$ \\
3 & $wxy+wz+xz+w+x+y$ \\
4 & $w^2+2wx+2wy+wxz+x^2+2xy+y^2$  \\
5 & $w^2x+w^2y+wx^2+3wxy+wy^2+wz+x^2y+xy^2+xz+yz$ \\
6 & $w^2xy+w^2z+wx^2y+wx+wxy^2+2wxz+wy+wyz$ \\
& $\qquad+x^2z+xyz+xy+z$ \\
\hline
\end{tabular}}

\medskip
{\bf Table~1}: $S(n;Z)$ for $0\leq n\leq 6$.
\end{center}

\bigskip
When $Z=Z_0:=(1,1,1,1)$, we set $S(n):=S(n;Z_0)$, and we compute
\begin{equation}\label{2.2}
\big(S(n)\big)_{n\geq 0}
=(1,3,4,{\bf 6},10,12,13,15,16,18,22,24,{\bf 28},36,40,42,\ldots).
\end{equation}
By expanding the right-hand side of \eqref{2.1} with $w=x=y=z=1$, it is not
difficult to see the following interpretation of the sequence $S(n)$.

\begin{proposition}\label{prop:2.2}
For $n\geq 1$, $S(n)$ is exactly the number of base-3 partitions of $n$ under 
the conditions \eqref{1.1}. 
\end{proposition}

We now see that the sequence \eqref{2.2} is consistent with
Examples~1 and~2; the corresponding values have been highlighted.

If we now expand the right-hand side of \eqref{2.1} with variables $w,x,y,z$ 
in place, it is not difficult to see that we have the following combinatorial
interpretation of the coefficients of the polynomials $S(n;Z)$.

\begin{proposition}\label{prop:2.3}
If we write $S(n;Z)$ in the form
\begin{equation}\label{2.3}
S(n;Z)=\sum_{i,j,k,\ell\geq 0}c(i,j,k,\ell)\,w^ix^jy^kz^\ell,
\end{equation}
then among all colored base-3 partitions of $n$, restricted by \eqref{1.1},
the coefficient $c(i,j,k,\ell)$ counts those that have
\begin{enumerate}
\item[] $i$ parts marked with an overline, and
\item[] $j$ parts marked with a tilde, and
\item[] $k$ single unmarked parts, and
\item[] $\ell$ pairs of unmarked parts.
\end{enumerate}
\end{proposition}

It is easy to see with Table~1 for $n=3$ that this is consistent with Example~1.
The case $n=12$ gives the following more meaningful example.

\medskip
\noindent
{\bf Example 3.} Computations show that one of the monomials of $S(12;Z)$ is
$2w^2xyz$. This corresponds to the two partitions
$(3\overline{3}\tilde{3}11\overline{1})$ and
$(33\overline{3}1\overline{1}\tilde{1})$ from Example~2, consistent with
Proposition~\ref{prop:2.3}.

Another monomial of $S(12;Z)$ is $3wxz$, which corresponds to the three
partitions $(\overline{9}11\tilde{1})$, $(\tilde{9}11\overline{1})$, and
$(33\overline{3}\tilde{3})$ from Example~2, again consistent with 
Proposition~\ref{prop:2.3}.

\medskip
Next we state a set of recurrence relations for the polynomials $S(n;Z)$. They
follow directly from Theorem~4.3 in \cite{DE10}.

\begin{proposition}\label{prop:2.4}
The polynomials $S(n;Z)$ satisfy the recurrence $S(0;Z)=1$, $S(1;Z)=w+x+y$, 
$S(2,Z)=wx+wy+xy+z$, and for $n\geq 1$,
\begin{align}
S(3n;Z)&=S(n;Z)+(wxy+wz+xz)\cdot S(n-1;Z),\label{2.4}\\
S(3n+1;Z)&=(w+x+y)\cdot S(n;Z)+wxz\cdot S(n-1;Z),\label{2.5}\\
S(3n+2;Z)&=(wx+wy+xy+z)\cdot S(n;Z).\label{2.6}
\end{align}
\end{proposition}

We finish this section with an easy consequence of Proposition~\ref{prop:2.4}.
We note that the sequence $s_k(n):=k\cdot3^n-1$, for a fixed $k\geq 1$,
satisfies the recurrence $s_k(n+1)=3\cdot s_k(n)+2$ with initial conditions 
$s_k(0)=k-1$. Therefore, upon iterating \eqref{2.6} and noting that the 
coefficient on the right of \eqref{2.6} is $S(2;Z)$, we get the following 
explicit formulas.

\begin{corollary}\label{cor:2.5}
For any integers $k\geq 1$ and $n\geq 0$, we have
\begin{equation}\label{2.7}
S(k\cdot 3^n-1,Z) = S(k-1;N)\cdot S(2;Z)^n,
\end{equation} 
and in particular
\begin{equation}\label{2.8}
S(3^n-1,Z) = (wx+wy+xy+z)^n.
\end{equation}
\end{corollary}

Upon setting $w=x=y=z=1$ in \eqref{2.8}, we get as an immediate consequence 
the fact that the number of base-3 partitions of $3^n-1$, restricted by
\eqref{1.1}, is $4^n$. Analogous enumerations can be obtained from \eqref{2.7}.

Corollary~\ref{cor:2.5} is related to a more interesting variant which we will 
now consider in Section~\ref{sec:3}.

\section{Two different subsequences}\label{sec:3}

In analogy to the subsequences considered in Corollary~\ref{cor:2.5}, we now
define the two polynomial sequences
\begin{equation}\label{3.1}
Q_n(Z):=S(\tfrac{3^n-3}{2};Z),\qquad R_n(Z):=S(\tfrac{3^n-1}{2};Z),
\end{equation}
where as before $Z=(w,x,y,z)$. For greater clarity and ease of notation we now
set
\begin{align}
W_1(Z)&:=wxy+wz+xz+w+x+y,\label{3.2}\\
W_2(Z)&:=w^2xy+w^2z+wx^2y+wxy^2+wxz+wyz+x^2z+xyz.\label{3.3}
\end{align}
Most of the properties of the polynomial sequences $Q_n(Z)$ and $R_n(Z)$
depend on the following results.

\begin{proposition}\label{prop:3.1}
$($a$)$ The polynomials $Q_n(Z)$ satisfy $Q_0(Z)=0$, $Q_1(Z)=1$, and for 
$n\geq 2$,
\begin{equation}\label{3.4}
Q_n(Z)=W_1(Z)\cdot Q_{n-1}(Z) - W_2(Z)\cdot Q_{n-2}(Z).
\end{equation}
$($b$)$ The polynomials $R_n(Z)$ satisfy $R_0(Z)=1$, $R_1(Z)=w+x+y$, and 
for $n\geq 2$,
\begin{equation}\label{3.5}
R_n(Z)=W_1(Z)\cdot R_{n-1}(Z) - W_2(Z)\cdot R_{n-2}(Z).
\end{equation}
\end{proposition}

\begin{proof}
We prove both identities jointly by induction on $n$. For greater ease of 
notation we suppress the multi-variable $Z$. We first verify the values
for $Q_1$, $R_0$ and $R_1$ with \eqref{3.1} and Table~1. We also note
that $Q_2=S(3)$, $R_2=S(4)$, and by direct computation using Table~1 we can
then verify that \eqref{3.4} and \eqref{3.5} holds for $n=2$. This is the 
induction beginning. 

We now assume that both \eqref{3.4} and \eqref{3.5} hold up to some $n-1$.
We wish to show that they then also hold for $n$, that is, as written in
\eqref{3.4} and \eqref{3.5}. Using the easy identities
\begin{equation}\label{3.6}
\frac{3^{n-j}-3}{2}=3\cdot\frac{3^{n-j-1}-1}{2},\qquad
\frac{3^{n-j-1}-1}{2}-1=\frac{3^{n-j-1}-3}{2}
\end{equation}
for $j=0,1,2$, we get with \eqref{2.4} that
\begin{align*}
Q_n-&W_1Q_{n-1}+W_2Q_{n-2}\\
&= S(\tfrac{3^n-3}{2})-W_1S(\tfrac{3^{n-1}-3}{2})+W_2S(\tfrac{3^{n-2}-3}{2})\\
&=S(\tfrac{3^{n-1}-1}{2})-W_1S(\tfrac{3^{n-2}-1}{2})+W_2S(\tfrac{3^{n-3}-1}{2})\\
&\quad+(wxy+wz+xz)\left(S(\tfrac{3^{n-1}-3}{2})-W_1S(\tfrac{3^{n-2}-3}{2})
+W_2S(\tfrac{3^{n-3}-3}{2})\right).
\end{align*}
Applying the two identities in \eqref{3.1} to the last two lines, we see by the
induction hypothesis that both lines vanish, which proves \eqref{3.4}.

Similarly, using \eqref{3.6} again repeatedly, we get with \eqref{2.5} that
\begin{align*}
R_n-&W_1R_{n-1}+W_2R_{n-2}\\
&= S(\tfrac{3^n-3}{2}+1)-W_1S(\tfrac{3^{n-1}-3}{2}+1)
+W_2S(\tfrac{3^{n-2}-3}{2}+1)\\
&=(w+x+y)\left(S(\tfrac{3^{n-1}-1}{2})-W_1S(\tfrac{3^{n-2}-1}{2})
+W_2S(\tfrac{3^{n-3}-1}{2})\right)\\
&\quad+wxz\left(S(\tfrac{3^{n-1}-1}{2}-1)-W_1S(\tfrac{3^{n-2}-1}{2}-1)
+W_2S(\tfrac{3^{n-3}-1}{2}-1)\right).
\end{align*}
Once again we apply the two identities in \eqref{3.1} to the last two lines and
see that both expressions in large parentheses vanish by the induction 
hypothesis. This proves \eqref{3.5} and the proof by induction is complete.
\end{proof}

Before we continue with the general case, we specialize $Z=Z_0:=(1,1,1,1)$ and
set $q_n:=Q_n(Z_0)$ and $r_n:=R_n(Z_0)$. Then Proposition~\ref{prop:3.1}
simplifies to
\begin{align}
q_0&=0, q_1=1,\quad\hbox{and}\quad q_n=6q_{n-1}-8q_{n-2}
\quad(n\geq 2),\label{3.7}\\
r_0&=1, r_1=3,\quad\hbox{and}\quad r_n=6r_{n-1}-8r_{n-2}
\quad(n\geq 2).\label{3.8}
\end{align}
The characteristic polynomial for the two recurrence relations is 
$x^2-6x+8=(x-4)(x-2)$, and it is not difficult to verify that the Binet-type
formulas are, for $n\geq 0$,
\begin{align}
q_n&=\frac{1}{2}\left(4^n-2^n\right) = 2^{n-1}\left(2^n-1\right),\label{3.9}\\
r_n&=\frac{1}{2}\left(4^n+2^n\right) = 2^{n-1}\left(2^n+1\right).\label{3.10}
\end{align}

With \eqref{3.9} and using the definition of $q_n$ and \eqref{3.1}, 
Proposition~\ref{prop:2.2} now implies that our observation following 
Example~2 is indeed true in general.

\begin{corollary}\label{cor:3.2}
For any $n\geq 1$, the number of base-3 partitions of $\frac{1}{2}(3^n-3)$
under the condition \eqref{1.1} is $2^{n-1}(2^n-1)$. In particular, if $n$
is a prime number such that $2^n-1$ is also prime, then the number of such
partitions is a perfect number. Conversely, all even perfect numbers can be
expressed in this way.
\end{corollary}

The second and third statements come from the well-known characterization of
even perfect numbers due to Euclid and Euler. If $n$ is a prime, $2^n-1$ is 
called a Mersenne number, which may or may not itself be prime. It is not
known whether there are infinitely many Mersenne primes.

In analogy to Corollary~\ref{cor:3.2}, the numbers $r_n$ lead to the following
consequence of Proposition~\ref{prop:2.2}.

\begin{corollary}\label{cor:3.3}
For any $n\geq 1$, the number of base-3 partitions of $\frac{1}{2}(3^n-1)$
under the condition \eqref{1.1} is $2^{n-1}(2^n+1)$.
\end{corollary}

We mention in passing that in the case $n=2^k$, $0\leq k\leq 4$, the Fermat
number $2^{2^k}+1$ is prime, but no other Fermat prime is known. It follows 
from a
famous result of Gauss that a regular polygon with $r_n=2^{n-1}(2^n+1)$ sides
is constructible with compass and straightedge if $n=2^k$, $0\leq k\leq 4$.

\medskip
Returning to the general polynomial case, we use Proposition~\ref{prop:3.1} to
derive generating functions for the polynomials $Q_n(Z)$ and $R_n(Z)$.

\begin{proposition}\label{prop:3.4}
The polynomials $Q_n(Z)$ and $R_n(Z)$ have the generating functions
\begin{align}
\sum_{n=0}^{\infty}Q_n(Z)q^n&=\frac{q}{1-W_1(Z)q+W_2(Z)q^2},\label{3.11}\\
\sum_{n=0}^{\infty}R_n(z)q^n
&=\frac{1-(wxy+wz+xz)q}{1-W_1(Z)q+W_2(Z)q^2}.\label{3.12}
\end{align}
\end{proposition}

\begin{proof}
We multiply both sides of \eqref{3.11} by the denominator on the right and
take the Cauchy product with the power series on the left. Then the constant
coefficient is $Q_0(Z)=0$, the coefficient of $q$ is $Q_1(Z)-W_1(Z)Q_0(z)=1$,
and all other coefficients are 0 by \eqref{3.4}. This shows that the numerator
on the right is $q$, which proves \eqref{3.11}. The identity \eqref{3.12} is 
proven analogously, using \eqref{3.5}.
\end{proof}

Propositions~\ref{prop:3.1} and~\ref{prop:3.4} indicate possible connections 
with the Chebyshev polynomials of the first and second kind, $T_n(v)$ and
$U_n(v)$, which can be defined by their generating functions

\begin{equation}\label{3.13}
\sum_{n=0}^{\infty}T_n(v)t^n = \frac{1-vt}{1-2vt+t^2},\qquad
\sum_{n=0}^{\infty}U_n(v)t^n = \frac{1}{1-2vt+t^2}.
\end{equation}
Numerous properties of these two polynomial sequences can be found, for 
instance, in \cite{Ri}. We can now state and prove the following identities,
which are similar in nature to those in Proposition~3.3 of \cite{DE11}.

\begin{proposition}\label{prop:3.5}
For all $n\geq 0$ we have
\begin{align}
Q_{n+1}(Z)&=W_2(Z)^{n/2}\cdot 
U_n\left(\frac{W_1(Z)}{2W_2(Z)^{1/2}}\right),\label{3.14}\\
R_n(Z)&=W_2(Z)^{n/2}\cdot 
T_n\left(\frac{W_1(Z)}{2W_2(Z)^{1/2}}\right)\label{3.15} \\
&\qquad\qquad\qquad +\frac{w+x+y-(wxy+wz+xz)}{2}Q_n(Z),\nonumber
\end{align}
with $W_1(Z)$ and $W_2(Z)$ as defined in \eqref{3.2} and \eqref{3.3}.
\end{proposition}

\begin{proof}
Comparing \eqref{3.11} with the second identity in \eqref{3.13}, we see that
\[
q=\frac{t}{W_2(Z)^{1/2}}\quad\hbox{and}\quad v=\frac{W_1(Z)}{2W_2(Z)^{1/2}}.
\]
Equating coefficients of $q^n$ then gives \eqref{3.14}.

Next, with $t$ and $v$ as above, we rewrite the numerator on the right of 
\eqref{3.12} as
\begin{equation}\label{3.16}
1-(wxy+wz+xz)q = 1-vt+\frac{w+x+y-wxy-wz-xz}{2W_2(Z)^{1/2}}t.
\end{equation}
The term $1-vt$, together with the first identity in \eqref{3.13}, leads to the
first summand on the right of \eqref{3.15}, while the remaining term on the 
right-hand side of \eqref{3.16} leads to the second summand on the right of 
\eqref{3.15}.
\end{proof}

The connections in Proposition~\ref{prop:3.5} between the polynomials $Q_n(Z)$,
$R_n(Z)$ and the Chebyshev polynomials will be particularly useful in the 
next section.

\section{Single-variable polynomials: The case $Z=(1,1,z,1)$}\label{sec:4}

There are numerous ways of specializing or equating the four variables in
$Z=(w,x,y,z)$ so that we obtain single-variable polynomials. In this and the
next section we concentrate on a few special cases which allow for reasonable 
combinatorial interpretations, while at the same time the polynomials in 
question have some interesting properties in their own right.

This section will be devoted to what we consider the most interesting of these
cases. For ease of notation we set $Z_1:=(1,1,z,1)$ and
\[
Q_n^{(1)}(z):=Q_n(Z_1),\qquad R_n^{(1)}(z):=R_n(Z_1).
\]
Then by \eqref{3.2} and
\eqref{3.3} we have
\begin{equation}\label{4.1}
W_1^{(1)}(z):=W_1(Z_1)=2(z+2),\qquad W_2^{(1)}(z):=W_2(Z_1)=(z+1)(z+2),
\end{equation}
and Proposition~\ref{prop:3.1}, with $R_1^{(1)}(z)=2+z$, allows us to 
efficiently compute the polynomials $Q_n^{(1)}(z)$ and $R_n^{(1)}(z)$; 
see Table~2.

Proposition~\ref{prop:2.3} together with \eqref{3.1} provides the following
combinatorial interpretations of the polynomials $Q_n^{(1)}(z)$ and 
$R_n^{(1)}(z)$.

\begin{proposition}\label{prop:4.1}
For $n\geq 1$, let
\begin{equation}\label{4.2}
Q_n^{(1)}(z) = \sum_{k=0}^{n-1}c_n(k)z^k,\qquad 
R_n^{(1)}(z) = \sum_{k=0}^{n}d_n(k)z^k.
\end{equation}
Then $c_n(k)$ counts the number of colored base-$3$ partitions of 
$\frac{1}{2}(3^n-3)$, restricted by \eqref{1.1} and having exactly $k$ single
unmarked powers of $3$.
Similarly, $d_n(k)$ counts the number of such partitions of 
$\frac{1}{2}(3^n-1)$.
\end{proposition}

\medskip
\begin{center}
{\renewcommand{\arraystretch}{1.1}
\begin{tabular}{|r|l|l|}
\hline
$n$ & $Q_n^{(1)}(z)$ & $R_n^{(1)}(z)$ \\
\hline
0 & 0 & 1 \\
1 & 1 & $z+2$ \\
2 & $2z+4$ & $z^2+4z+5$ \\
3 & $3z^2+12z+13$ & $z^3+6z^2+15z+14$ \\
4 & $4z^3+24z^2+52z+40$ & $z^4+8z^3+30z^2+56z+41$ \\
5 & $5z^4+40z^3+130z^2+200z+121$ & $z^5+10z^4+50z^3+140z^2+205z+122$ \\
\hline
\end{tabular}}

\medskip
{\bf Table~2}: $Q_n^{(1)}(z)$ and $R_n^{(1)}(z)$ for $0\leq n\leq 5$.
\end{center}
\bigskip

Table~2 shows some obvious patterns in the leading and constant coefficients
of the two polynomial sequences. Indeed, some straightforward inductions show
that
\begin{equation}\label{4.3}
c_n(0)=\frac{3^n-1}{2},\quad c_n(n-1)=n,\quad\hbox{and}\quad
d_n(0)=\frac{3^n+1}{2},\quad d_n(n)=1.
\end{equation}
We will see below that these are special cases of a more general result.

\medskip
Next, we derive some special properties of the
polynomials $Q_n^{(1)}(z)$ and $R_n^{(1)}(z)$. We first observe that with 
$w=x=z=1$ in \eqref{3.15}, the summand on the right of \eqref{3.15} vanishes, 
and with \eqref{4.1} we get
\begin{align}
Q_{n+1}^{(1)}(z)&=\big((z+1)(z+3)\big)^{n/2}\cdot
U_n\left(\frac{z+2}{\sqrt{(z+1)(z+3)}}\right),\label{4.4}\\
R_n^{(1)}(z)&=\big((z+1)(z+3)\big)^{n/2}\cdot
T_n\left(\frac{z+2}{\sqrt{(z+1)(z+3)}}\right).\label{4.5}
\end{align}

All further properties are consequences of these two identities. We begin with
the following rather surprising relations.

\begin{proposition}\label{prop:4.2}
For all $n\geq 0$,
\begin{align}
R_n^{(1)}(z)-Q_n^{(1)}(z) &= (z+1)^n,\label{4.6}\\
R_n^{(1)}(z)^2-Q_n^{(1)}(z)^2 &= \left((z+1)(z+3)\right)^n,\label{4.7}\\
R_n^{(1)}(z)+Q_n^{(1)}(z) &= (z+3)^n.\label{4.8}
\end{align}
\end{proposition}

The identity \eqref{4.6} is somewhat similar to (4.4) in \cite{DE11}, in the
proof of which the identity \eqref{4.9} below was also derived and used. 

\begin{proof}[Proof of Proposition~\ref{prop:4.2}]
We use the defining identities
\[
\sin{\theta}\cdot U_{n-1}(\cos{\theta})=\sin(n\theta),\qquad
T_n(\cos{\theta})=\cos(n\theta).
\]
Multiplying both sides of the left identity by $i$, then adding both and using
$2i\sin{\theta}=e^{i\theta}-e^{-i\theta},\,
2\cos{\theta}=e^{i\theta}+e^{-i\theta}$, and $u:=e^{i\theta}$, we get
\begin{equation}\label{4.9}
\frac{u-u^{-1}}{2}U_{n-1}\left(\frac{u+u^{-1}}{2}\right)
+T_n\left(\frac{u+u^{-1}}{2}\right) = u^n,\qquad n\geq 1.
\end{equation}
We now set $u=\sqrt{(z+1)/(z+3)}$ and verify that
\[
\frac{u-u^{-1}}{2}=\frac{-1}{\sqrt{(z+1)(z+3)}},\qquad
\frac{u+u^{-1}}{2}=\frac{z+2}{\sqrt{(z+1)(z+3)}}.
\] 
Upon substituting these two identities into \eqref{4.9} and multiplying both 
sides by $((z+1)(z+3))^{n/2}$, we get
\begin{align}
-\big((z+1)&(z+3)\big)^{(n-1)/2}\cdot
U_{n-1}\left(\frac{z+2}{\sqrt{(z+1)(z+3)}}\right)\label{4.10}\\
&+\big((z+1)(z+3)\big)^{n/2}\cdot
T_n\left(\frac{z+2}{\sqrt{(z+1)(z+3)}}\right)=(z+1)^n.\nonumber
\end{align}
Finally, with \eqref{4.4} and \eqref{4.5} we see that \eqref{4.10} is 
equivalent to \eqref{4.6}.

To prove \eqref{4.7}, we use the well-known Pell-type equation
\[
T_n(x)^2-(x^2-1)U_{n-1}(x)^2=1
\]
for the Chebyshev polynomials (see, e.g., \cite[p.~9]{Ri}), which in our case
becomes
\begin{equation}\label{4.11}
T_n\left(\frac{z+2}{\sqrt{(z+1)(z+3)}}\right)^2
-\frac{1}{(z+1)(z+3)}U_{n-1}\left(\frac{z+2}{\sqrt{(z+1)(z+3)}}\right)^2=1.
\end{equation}
Upon multiplying both sides of \eqref{4.11} by $((z+1)(z+3))^n$, we see that
\eqref{4.7} immediately follows from \eqref{4.4} and \eqref{4.5}. Finally, 
the identity \eqref{4.8} follows from dividing \eqref{4.7} by \eqref{4.6}.
\end{proof}

By adding, respectively subtracting, \eqref{4.6} and \eqref{4.8}, we obtain
the following consequences.

\begin{corollary}\label{cor:4.3}
For all $n\geq 0$ we have
\begin{align}
Q_n^{(1)}(z) &= \frac{(z+3)^n-(z+1)^n}{2}
= \sum_{j=0}^n\binom{n}{j}\frac{3^{n-j}-1}{2}\,z^j,\label{4.12}\\
R_n^{(1)}(z) &= \frac{(z+3)^n+(z+1)^n}{2}
= \sum_{j=0}^n\binom{n}{j}\frac{3^{n-j}+1}{2}\,z^j.\label{4.13}
\end{align}
\end{corollary}

Comparing Corollary~\ref{cor:4.3} with the notation in \eqref{4.2}, we see that
the identities in \eqref{4.3} are special cases of \eqref{4.12} and 
\eqref{4.13}. Proposition~\ref{prop:4.1} combined with Corollary~\ref{cor:4.3}
now leads to the following combinatorial interpretations.

\begin{corollary}\label{cor:4.4}
Considering colored base-$3$ partitions restricted by \eqref{1.1}, for each
$j=0,1,\ldots,n$,

$(a)$ there are $\frac{1}{2}\binom{n}{j}(3^{n-j}-1)$ partitions of 
$\frac{1}{2}(3^n-3)$ that have exactly $j$ single unmarked parts.

$(b)$ there are $\frac{1}{2}\binom{n}{j}(3^{n-j}+1)$ partitions of 
$\frac{1}{2}(3^n-1)$ that have exactly $j$ single unmarked parts. 
\end{corollary}

\medskip
\noindent
{\bf Example~4.} (a) For $n=2$, the six partitions of $\frac{1}{2}(3^2-3)=3$ are
given in Example~1 (note the different use of $n$). Four partitions have no
single unmarked part, two have exactly one, and there is no partition with
two unmarked parts, consistent with Corollary~\ref{cor:4.4}(a).

(b) Taking again $n=2$, the ten partitions of $\frac{1}{2}(3^2-1)=4$ are, in
the notation of Example~2:
\[
(31),(3\overline{1}),(3\tilde{1}),
(\overline{3}1),(\overline{3}\overline{1}),(\overline{3}\tilde{1}),
(\tilde{3}1),(\tilde{3}\overline{1}),(\tilde{3}\tilde{1}),
(11\overline{1}\tilde{1}).
\]
We see that $\frac{1}{2}(3^2+1)=5$ of these partitions have no single
unmarked parts, while $\frac{1}{2}\binom{2}{1}(3^1+1)=4$ have one unmarked 
part and one partition, namely $(31)$, has two unmarked parts.

\medskip
We conclude this section with a few more properties of the polynomials
$Q_n^{(1)}(z)$ and $R_n^{(1)}(z)$. First, upon replacing $z$ by $z-2$ in the
first equalities of \eqref{4.12} and \eqref{4.13}, we get the following simple
representations.

\begin{corollary}\label{cor:4.5}
For all $n\geq 0$, 
\begin{equation}\label{4.14}
Q_n^{(1)}(z-2)
=\sum_{j=0}^{\lfloor\frac{n-1}{2}\rfloor}\binom{n}{2j+1}z^{n-2j-1},\quad
R_n^{(1)}(z-2)
=\sum_{j=0}^{\lfloor\frac{n}{2}\rfloor}\binom{n}{2j}z^{n-2j}.
\end{equation}
\end{corollary}

Concerning a related shift, it is a consequence of the identity \eqref{4.6} 
that $R_n^{(1)}(z-1)$
and $Q_n^{(1)}(z-1)$ differ by only the leading term $z^n$ of $R_n^{(1)}(z-1)$.
Also, as mentioned in entry A082137 of \cite{OEIS}, the coefficients of this 
last polynomial appear in \cite[Table~1]{LU} in a different combinatorial 
setting.

We can also determine the zero distribution of the polynomials in question.

\begin{proposition}\label{prop:4.6}
All the zeros of the polynomials $Q_n^{(1)}(z)$ and $R_n^{(1)}(z)$ lie on the
vertical line ${\rm Re}(z)=-2$. The zeros are all nonreal with the exception
of $Q_{2n}^{(1)}(-2)=0$ and $R_{2n+1}^{(1)}(-2)=0$ for all $n\geq 0$.
\end{proposition}

\begin{proof}
We replace $z$ by $z-2$ in \eqref{4.4} and \eqref{4.5} and recall that the
zeros of $U_n(v)$ and $T_n(v)$ are all real and satisfy $-1<v<1$. This means
that we need to solve the equation $v=z/\sqrt{z^2-1}$ for $z$, and we get
\begin{equation}\label{4.15}
z = \pm i\sqrt{\frac{v^2}{1-v^2}}.
\end{equation}
Since the expression in the square root is real and positive, this proves the
first statement of the proposition. The second statement follows from 
Corollary~\ref{cor:4.5}.
\end{proof}

\section{Single-variable polynomials: Some further cases}\label{sec:5}

In this section we work out the details of two further cases, and in the end we
identify six more cases that could be done by using the same methods. In all 
cases we only consider the polynomials $Q_n(Z)$. Analogous results could also 
be obtained with the polynomials $R_n(Z)$ by making some necessary adjustments
such as using \eqref{3.5} and the second part of \eqref{3.1}. One main 
difference is that, in contrast to the situation in Section~\ref{sec:4}, the
additional summand in \eqref{3.15} makes it difficult, if not impossible, to 
determine the distribution of zeros of $R_n(Z)$.

\subsection{The case $Z=Z_2:=(z,z,z,z^2)$} 

In analogy to Section~\ref{sec:4}, we set 
\[
Q_n^{(2)}(z):=Q_n(Z_2).
\]
By \eqref{3.2} and \eqref{3.3} we have
\begin{equation}\label{5.1}
W_1^{(2)}(z):=W_1(Z_2)=3z(z^2+1),\qquad W_2^{(2)}(z):=W_2(Z_2)=8z^4,
\end{equation}
and with Proposition~\ref{prop:3.1} we can easily compute the entries in
Table~2.

\bigskip
\begin{center}
{\renewcommand{\arraystretch}{1.1}
\begin{tabular}{|r|l|}
\hline
$n$ & $Q_n^{(2)}(z)$ \\
\hline
1 & 1 \\
2 & $3z^3 + 3z$ \\
3 & $9z^6 + 10z^4 + 9z^2$ \\
4 & $27z^9 + 33z^7 + 33z^5 + 27z^3$ \\
5 & $81z^{12} + 108z^{10} + 118z^8 + 108z^6 + 81z^4$ \\
6 & $243z^{15} + 351z^{13} + 414z^{11} + 414z^9 + 351z^7 + 243z^5$ \\
7 & $729z^{18}+1134z^{16}+1431z^{14}+1540z^{12}+1431z^{10}+1134z^8+729z^6$ \\
\hline
\end{tabular}}

\medskip
{\bf Table~3}: $Q_n^{(2)}(z)$ for $1\leq n\leq 7$.
\end{center}
\bigskip

Before dealing with combinatorial interpretations of the polynomials 
$Q_n^{(2)}(z)$, we prove some basic properties that are apparent from Table~3.
As before, the connection with Chebyshev polynomials will be an important tool.
Substituting the identities from \eqref{5.1} into \eqref{3.14}, we get
\begin{equation}\label{5.2}
Q_{n+1}^{(2)}(z)=\left(2\sqrt{2}z^2\right)^n\cdot
U_n\left(\tfrac{3}{4\sqrt{2}}(z+z^{-1})\right).
\end{equation}
We can now state and prove the following properties.

\begin{proposition}\label{prop:5.1}
For $n\geq 0$, the polynomial $Q_{n+1}^{(2)}(z)$ has degree $3n$, its term with
lowest degree has degree $n$, it is palindromic, and its leading coefficient
is $3^n$. Furthermore, when $n$ is even (resp.\ odd), $Q_{n+1}^{(2)}(z)$ has
only even (resp.\ odd) powers of $z$.
\end{proposition}

\begin{proof}
With \eqref{3.4} and \eqref{5.1} we have the recurrence $Q_0^{(2)}(z)=0$,
$Q_1^{(2)}(z)=1$, and for $n\geq 2$, 
\begin{equation}\label{5.3}
Q_n^{(2)}(z)=3z(z^2+1)Q_{n-1}^{(2)}(z)-8z^4Q_{n-2}^{(2)}(z).
\end{equation}
Then the modified polynomials, which we define by $\widetilde{Q}_0^{(2)}(z)=0$ 
and
\begin{equation}\label{5.4}
\widetilde{Q}_{n+1}^{(2)}(z):=z^{-n}Q_{n+1}^{(2)}(z)\quad (n\geq 0),
\end{equation}
satisfy the recurrence given by $\widetilde{Q}_1^{(2)}(z)=1$ and for $n\geq 2$,
\begin{equation}\label{5.5}
\widetilde{Q}_n^{(2)}(z)=3(z^2+1)\widetilde{Q}_{n-1}^{(2)}(z)
-8z^2\widetilde{Q}_{n-2}^{(2)}(z).
\end{equation}
Now an easy induction using \eqref{5.5} and the initial terms shows that 
$\deg{\widetilde{Q}_{n+1}^{(2)}(z)}=2n$ and that the leading coefficient of
$\widetilde{Q}_{n+1}^{(2)}(z)$ is $3^n$.

Next, by \eqref{5.2} and \eqref{5.4} we have
\begin{equation}\label{5.6}
\widetilde{Q}_{n+1}^{(2)}(z)=\left(2\sqrt{2}z\right)^n\cdot
U_n\left(\tfrac{3}{4\sqrt{2}}(z+z^{-1})\right),
\end{equation}
which immediately shows that
\[
z^{2n}\widetilde{Q}_{n+1}^{(2)}(\tfrac{1}{z}) = \widetilde{Q}_{n+1}^{(2)}(z),
\]
which means that $\widetilde{Q}_{n+1}^{(2)}(z)$ is a reciprocal (or palindromic)
polynomial. Finally, the statement concerning the parity of the powers follows
from \eqref{5.3} by an easy induction. This completes the proof of the 
proposition.
\end{proof}

With Proposition~\ref{prop:2.3} and \eqref{3.1}, along with 
Proposition~\ref{prop:5.1}, we can now state the following properties.

\begin{proposition}\label{prop:5.2}
For $n\geq 1$, let
\[
Q_n^{(2)}(z) = \sum_{k=n-1}^{3n-3}e_n(k)z^k.
\]
Then $e_n(k)$ counts the number of colored base-$3$ partitions of
$\frac{1}{2}(3^n-3)$, restricted by \eqref{1.1} and having a total of $k$ parts.
Furthermore,
\begin{enumerate}
\item[(a)] the number of parts lies between $n-1$ and $3n-3$,
\item[(b)] the number of parts cannot have the same parity as $n$,
\item[(c)] for $0\leq j\leq\lfloor(n-1)/2\rfloor$, there are as many partitions
with $n-1+j$ parts as there are with $3n-3-j$ parts,
\item[(d)] there are $3^{n-1}$ partitions with $n-1$ parts.
\end{enumerate}
\end{proposition}

\noindent
{\bf Example~5.} (1) Let $n=2$. Then $Q_2^{(2)}(z)=3z^3+3z$, consistent with
Example~1 which shows 3 partitions with three parts and 3 partitions with
one part.

(2) Similarly, $Q_3^{(2)}(z)=9z^6+10z^4+9z^2$ is consistent with Example~2,
which shows 9 partitions each with six and with two parts, and 10 partitions
with four parts. 

\medskip
The polynomials $Q_n^{(2)}(z)$ also have an interesting zero distribution, as
we will now see.

\begin{proposition}\label{prop:5.3}
The zeros of all polynomials $Q_n^{(2)}(z)$, $n\geq 2$, lie on the two segments
of the unit circle that satisfy $|\rm{Im}(z)|>\frac{1}{3}.$
\end{proposition}

\begin{proof}
By \eqref{5.2}, $z$ is a zero of $Q_{n+1}^{(2)}(z)$ if and only if
\begin{equation}\label{5.7}
v = \frac{3}{4\sqrt{2}}\left(z+\frac{1}{z}\right).
\end{equation}
Solving \eqref{5.7} for $z$, we obtain
\begin{equation}\label{5.8}
z = \frac{2\sqrt{2}}{3}v\pm i\sqrt{1-\tfrac{8}{9}v^2}
\end{equation}
and verify
\[
|z|^2 =\left(\tfrac{2\sqrt{2}}{3}v\right)^2+1-\tfrac{8}{9}\sqrt{2}v^2 = 1,
\]
so the zeros of $Q_{n+1}^{(2)}(z)$ lie on the unit circle. Finally, since the
zeros of $U_n(v)$ satisfy $-1<v<1$, \eqref{5.8} shows that the zeros of 
$Q_{n+1}^{(2)}(z)$ have imaginary parts $>1/3$ or $<-1/3$, which completes 
the proof.
\end{proof}

\subsection{The case $Z=Z_3:=(1,1,z,z)$}

Here we set 
\[
Q_n^{(3)}(z):=Q_n(Z_3).
\]
By again using \eqref{3.2} and \eqref{3.3} we have
\begin{equation}\label{5.1a}
W_1^{(3)}(z):=W_1(Z_3)=4z+2,\qquad W_2^{(3)}(z):=W_2(Z_3)=3z^2+5z.
\end{equation}
Then Proposition~\ref{prop:3.1} gives the recurrence $Q_0^{(3)}(z)=0$,
$Q_1^{(3)}(z)=1$, and for $n\geq 2$,
\begin{equation}\label{5.2a}
Q_n^{(3)}(z)=(4z+2)Q_{n-1}^{(3)}(z)-(3z^2+5z)Q_{n-2}^{(3)}(z).
\end{equation}
With this, we can first compute the entries in Table~4.

\bigskip
\begin{center}
{\renewcommand{\arraystretch}{1.1}
\begin{tabular}{|r|l|}
\hline
$n$ & $Q_n^{(3)}(z)$ \\
\hline
1 & 1 \\
2 & $4z + 2$ \\
3 & $13z^2 + 11z + 4$ \\
4 & $40z^3 + 44z^2 + 28z + 8$ \\
5 & $121z^4 + 158z^3 + 133z^2 + 68z + 16$ \\
6 & $364z^5 + 542z^4 + 544z^3 + 374z^2 + 160z + 32$ \\
7 & $1093z^6 + 1817z^5 + 2071z^4 + 1715z^3 + 1000z^2 + 368z + 64$ \\
\hline
\end{tabular}}

\medskip
{\bf Table~4}: $Q_n^{(3)}(z)$ for $1\leq n\leq 7$.
\end{center}
\bigskip

We now prove some properties of the polynomials $Q_n^{(3)}(z)$ that are 
apparent from Table~4.

\begin{proposition}\label{prop:5.4}
For any $n\geq 1$, polynomial $Q_n^{(3)}(z)$ has the following properties:
\begin{enumerate}
\item[(a)] The degree of $Q_n^{(3)}(z)$ is $n-1$.
\item[(b)] The leading coefficient is $\frac{1}{2}(3^n-1)$.
\item[(c)] The constant coefficient is $2^{n-1}$.
\end{enumerate}
\end{proposition}

\begin{proof}
We prove parts (a) and (b) jointly, using \eqref{5.2a}. First, it is clear from
the recurrence that the degree of $Q_n^{(3)}(z)$ cannot exceed $n-1$. Second,
\eqref{5.2a} gives the following recurrence for the leading coefficients
$\ell_n$: $\ell_0=0$, $\ell_1=1$, and $\ell_n=4\ell_{n-1}-3\ell_{n-2}$ 
$(n\geq 2)$. It is now easy to verify that the sequence $\frac{1}{2}(3^n-1)$
satisfies the same recurrence. This proves parts (a) and (b).

Finally, \eqref{5.2a} also gives $Q_n^{(3)}(0)=2Q_{n-1}^{(3)}(0)$. This,
together with $Q_1^{(3)}(0)=1$, proves part (c).
\end{proof}

Once again, with \eqref{3.1} and Proposition~\ref{prop:2.3} we obtain a 
combinatorial interpretation. Here we also use Proposition~\ref{prop:5.4}.

\begin{proposition}\label{prop:5.5}
For $n\geq 1$, let
\[
Q_n^{(3)}(z) = \sum_{k=0}^{n-1}f_n(k)z^k.
\]
Then $f_n(k)$ counts the number of colored base-$3$ partitions of
$\frac{1}{2}(3^n-3)$, restricted by \eqref{1.1} and such that the number of
single unmarked parts plus the number of pairs of unmarked parts equals $k$.
Furthermore,
\begin{enumerate}
\item[(a)] there are $2^{n-1}$ partitions without unmarked parts,
\item[(b)] the maximal sum of the number of single unmarked parts and the 
number of pairs of unmarked parts in a partition is $n-1$, 
\item[(c)] this maximum is achieved by $\frac{1}{2}(3^n-1)$ partitions.
\end{enumerate}
\end{proposition}

\noindent
{\bf Example~6.} Let $n=2$. Then we see in Example~1 that two partitions have
no unmarked parts, while the remaining four partitions have either one single
unmarked part or one pair of unmarked parts. This is consistent with
$Q_2^{(3)}(z)=4z+2$.

\medskip
We will now show that the polynomials $Q_n^{(3)}(z)$ have particularly 
interesting zero distributions.

\begin{proposition}\label{prop:5.6}
The zeros of all polynomials $Q_n^{(3)}(z)$, $n\geq 2$, lie on the segment of
the circle with radius $\frac{7}{8}$ and centered at $z=\frac{3}{8}$ that 
satisfies Re$(z)<\frac{1}{2}$.
\end{proposition}

\begin{proof}
Once again, we use the connection with Chebyshev polynomials given by 
\eqref{3.14}. With \eqref{5.1a} we have
\[
Q_{n+1}^{(3)}(z)=\left(3z^2+5z\right)^{n/2}\cdot
U_n\left(\tfrac{2z+1}{\sqrt{3z^2+5z}})\right).
\]
This means that $z$ is a zero of $Q_{n+1}^{(3)}(z)$ if and only if
\begin{equation}\label{5.3a}
v=\frac{2z+1}{\sqrt{3z^2+5z}}
\end{equation}
is a zero of $U_n(v)$. We know that $-1<v<1$, and with the aim of solving
\eqref{5.3a} for $z$, we get the quadratic equation
\[
(4-3u)z^2 + (4-5u)z + 1 = 0,\qquad 0\leq u<1,
\]
where for simplicity we have set $u=v^2$. The quadratic formula now gives the
solutions
\begin{equation}\label{5.4a}
z=-\frac{4-5u}{8-6u}\pm i\sqrt{\frac{28u-25u^2}{(8-6u)^2}},
\end{equation}
or equivalently,
\[
z-\frac{3}{8}=\frac{29u-28}{4(8-6u)}\pm i\sqrt{\frac{28u-25u^2}{(8-6u)^2}}.
\]
We note that the term inside the square root is always positive for
$0\leq u\leq 1$. Then it is easy to verify that 
\[
\left(\frac{29u-28}{4(8-6u)}\right)^2+\frac{28u-25u^2}{(8-6u)^2}
= \left(\frac{7}{8}\right)^2.
\]
This, with \eqref{5.4a}, shows that the zeros lie on the circle in question.

To prove the remaining statement, we let $r(u)$ be the real part on the right
of \eqref{5.4a}. Then $r(0)=-1/2$, $r(1)=1/2$, and since the derivative
$r'(u)=4/(4-3u)^2$ is always positive, $r(1)=1/2$ is in fact the supremum of
the real parts of all the zeros of $Q_{n+1}^{(3)}(z)$. This is due to the fact
that the zeros of $U_n(v)$ get arbitrarily close to 1 and $-1$ as $n$ grows, 
but $\pm 1$ are not zeros themselves. This completes the proof.
\end{proof}

\subsection{Further cases}
Without going into any detail, we now summarize a few additional cases that 
allow for some meaningful combinatorial interpretations. The polynomials in 
question are $Q_n(Z)$, as defined in \eqref{3.1}. We write them generically as
\[
Q_n(Z) = \sum_{k\geq 0}a_n(k)z^k.
\]
Furthermore, we refer to the colored base-3 partitions of $\frac{1}{2}(3^n-3)$,
restricted by \eqref{1.1}, simply as ``partitions". The methods used in 
Section~\ref{sec:4} and in the previous two subsection can then be applied to
the following cases.

\begin{enumerate}
\item[(1)] $Z=(z,z,1,1)$:
$a_n(k)$ counts the number of partitions that have $k$ marked parts. $Q_n(Z)$
is palindromic.
\item[(2)] $Z=(z,z,z,z)$:
$a_n(k)$ counts the number of partitions that have $k$ parts, but with each
pair of unmarked parts counted only once.
\item[(3)] $Z=(1,1,z,z^2)$:
$a_n(k)$ counts the number of partitions that have a total of $k$ unmarked
parts. $Q_n(Z)$ is palindromic and its zeros lie on the unit circle.
\item[(4)] $Z=(z,z,z,1)$: $a_n(k)$ counts the number of partitions whose total
number of parts, excluding pairs of unmarked parts, is $k$.
\item[(5)] $Z=(1,z,z,z^2)$: $a_n(k)$ counts the number of partitions that have
$k$ parts that don't have an overline. $Q_n(Z)$ is palindromic and its zeros 
lie on the unit circle or on a finite segment of the negative real axis. 
\item[(6)] $Z=(z,1,z,z^2)$: This is the same as in (5), only with ``overline" 
replaced by ``tilde".
\end{enumerate}

\section{Further Remarks}\label{sec:6}

\subsection{Some polynomial identities}
We begin this section with a few more properties of the polynomials 
introduced in \eqref{3.1}. They all connect the two polynomial sequences with
each other, but do not seem to have any obvious combinatorial interpretations.

\begin{proposition}\label{prop:6.1}
For all $n\geq 1$ we have
\begin{align}
wxz\cdot Q_n(Z) &= R_{n+1}(Z) - (w+x+y)\cdot R_n(Z),\label{6.1}\\
R_n(Z) &= Q_{n+1}(Z) -(wxy+wz+xz)\cdot Q_n(Z).\label{6.2}
\end{align}
\end{proposition}

\begin{proof}
Upon replacing $n$ by $\frac{1}{2}(3^n-1)$ in \eqref{2.5} we get 
\eqref{6.1} immediately from \eqref{3.1}. With the same substitution in 
\eqref{2.4}, we also get \eqref{6.2}.
\end{proof}

If we consider \eqref{6.1} and \eqref{6.2} as modified difference equations,
it becomes clear that in both cases we can use telescoping. In fact, without
too much effort one obtains the following identities.

\begin{corollary}\label{cor:6.2}
For all $N\geq 1$ we have
\begin{align}
R_N(Z) &= (w+x+y)^N+wxz\cdot \sum_{n=1}^N (w+x+y)^{N-n}Q_{n-1}(Z),\label{6.3}\\
Q_N(Z) &= \sum_{n=1}^{N}(wxy+wz+xz)^{N-n}R_{n-1}(Z).\label{6.4}
\end{align}
\end{corollary}

\subsection{Explicit zeros}
The Chebyshev polynomials $U_n(v)$ and $T_n(v)$ are known to have the zeros
\begin{equation}\label{6.5}
v_k=\cos\left(\tfrac{k+1}{n+1}\pi\right),\quad\hbox{resp.}\quad
v_k=\cos\left(\tfrac{2k+1}{2n}\pi\right),\quad k=0,1,\ldots,n-1;
\end{equation}
see, e.g., \cite[pp.~6--7]{Ri}. Substituting these values into \eqref{4.15}, 
we get explicit formulas for all the zeros of $Q_n^{(1)}(z)$ and $R_n^{(1)}(z)$.
Similarly, substituting the values of the first part of \eqref{6.5} into
\eqref{5.8} and \eqref{5.4a}, recalling that $u=v^2$ in this last case, we 
get explicit expressions also for the zeros of $Q_n^{(2)}(z)$ and 
$Q_n^{(3)}(z)$.

\subsection{Divisibility results}
Since the polynomials we have dealt with in this paper all have integer 
coefficients, it makes sense to ask about divisibility and irreducibility 
over the rationals. By \eqref{3.14}, divisibility properties of the 
Chebyshev polynomials $U_n(v)$ carry over to $Q_n(Z)$. In particular, since
$U_{n-1}(v)$ is a divisibility sequence (see, e.g., \cite[pp.~227 ff]{Ri}),
we have the following result.

\begin{corollary}\label{cor:6.3}
For $j=1,2,3$, the polynomial sequences $Q_n^{(j)}(z)$ are divisibility
sequences, that is, if $m|n$ then $Q_m^{(j)}(z)|Q_n^{(j)}(z)$. 
\end{corollary}

\noindent
{\bf Example~7.} The polynomial $Q_6^{(3)}(z)$ in Table~4 factors as
$2(2z+1)(7z^2+z+4)(13z^2+11z+4)$, and we see that $Q_2^{(3)}(z)$ and
$Q_3^{(3)}(z)$ are among the factors.

\bigskip
In general, the situation is less straightforward for the polynomials $R_n(Z)$,
but for $R_n^{(1)}(z)$ the additional term in \eqref{3.15} disappears. The
identity \eqref{4.5} then shows that divisibility properties of $T_n(v)$ carry
over to $R_n^{(1)}(z)$. However, in contrast to $U_n(v)$, the polynomials
$T_n(v)$ are only a partial divisibility sequence. For instance,
$R_6^{(1)}(z)=(z^2+4z+5)(z^4+8z^3+38z^2+88z+73)$, and we see from Table~2 that
$R_2^{(1)}(z)$ divides $R_6^{(1)}(z)$, but $R_1^{(1)}(z)$ and $R_3^{(1)}(z)$ 
do not.

\medskip
Finally, since compositions of functions are involved in \eqref{3.14}, we 
cannot directly conclude that known irreducibility results for Chebyshev 
polynomials will carry 
over to our $Q$- and $R$-polynomials. For some related approaches and results,
see Proposition~6.8 in \cite{DE11}.


\begin{thebibliography}{25}

\bibitem{DE10} K.~Dilcher and L.~Ericksen, Polynomial analogues of restricted 
multicolor $b$-ary partition functions, {\it Int. J. Number Theory} {\bf 17}
(2021), no.~2, 371--391.

\bibitem{DE11} K.~Dilcher and L.~Ericksen, Polynomials and algebraic curves 
related to certain binary and $b$-ary overpartitions, {\it Ramanujan J.}
{\bf 68}, 44 (2025).\\
{\tt https://doi.org/10.1007/s11139-025-01187-3}.

\bibitem{Ke} W.~J.~Keith, Restricted $k$-color partitions, {\it Ramanujan J.}
{\bf 40} (2016), 71--92.

\bibitem{LU} A.~Laradji and A.~Umar, Combinatorial results for semigroups of 
order-preserving partial transformations, {\it J. Algebra} {\bf 278} (2004), 
no.~1, 342--359.

\bibitem{OEIS}
OEIS Foundation Inc. (2011), {\it The On-Line Encyclopedia of Integer 
Sequences},\\
{\tt http://oeis.org}.\rm

\bibitem{Ri} T.~J.~Rivlin, {\it Chebyshev Polynomials. From Approximation 
Theory to Algebra and Number Theory. Second Edition}. Wiley, New York, 1990.

\bibitem{UZ} M.~Ulas and B.~{\.Z}mija, On $p$-adic valuations of colored 
$p$-ary partitions, {\it Monatsh. Math.} {\bf 188} (2019), no.~2, 351--368. 

\end{thebibliography}
\end{document}